\documentclass[a4paper, 12pt]{article}
\usepackage{mathrsfs}
\usepackage{amsmath}
\usepackage{amsfonts}
\usepackage[T1]{fontenc}
\usepackage[latin9]{inputenc}
\usepackage{amsthm}
\usepackage{graphicx}
\usepackage{amsmath}

\usepackage{amsthm}
\usepackage{array}
\usepackage{cases}
\usepackage{enumerate}
\usepackage[natural]{xcolor}
\usepackage[ruled,vlined,english,linesnumbered]{algorithm2e}
\usepackage{enumerate,enumitem,todonotes}
\makeatletter
\def\th@plain{%
  \itshape 
}
\makeatother

\makeatletter
\renewenvironment{proof}[1][\proofname]{\par
  \pushQED{\qed}%
  \normalfont \topsep6\p@\@plus6\p@\relax
  \trivlist
  \item[\hskip\labelsep
        \bfseries
    #1\@addpunct{.}]\ignorespaces
}{%
  \popQED\endtrivlist\@endpefalse
}
\makeatother

\newtheorem{theorem}{Theorem}[section]
\newtheorem{observation}[theorem]{Observation}

\numberwithin{equation}{section}

\newtheorem{thm}{Theorem}[section]
\newtheorem{cor}[thm]{Corollary}

\newtheorem{lem}[thm]{Lemma}
\newtheorem{prop}[thm]{Proposition}

\newtheorem{conj}[thm]{Conjecture}

\numberwithin{equation}{section}

\usepackage[varg]{txfonts}

\usepackage{graphicx, epsfig, subfigure}
\usepackage{enumerate} 
\usepackage[square, numbers, sort&compress]{natbib} 
\ifx\pdfoutput\undefined
 \usepackage[dvipdfm,%
  pdfstartview=FitH, 
  bookmarks=true,%
  bookmarksnumbered=true, 
  bookmarksopen=true, 
  plainpages=false,%
  pdfpagelabels,%
  colorlinks=true, 
  linkcolor=blue, 
  citecolor=blue,%
  urlcolor=black,
  pdfborder=001]{hyperref}
  \AtBeginDvi{}  
\else
 \usepackage[pdftex,%
  pdfstartview=FitH, 
  bookmarks=true,%
  bookmarksnumbered=true, 
  bookmarksopen=true, 
  plainpages=false,%
  pdfpagelabels,%
  colorlinks=true, 
  linkcolor=blue, 
  citecolor=blue,%
  urlcolor=black,
  pdfborder=001]{hyperref}
\fi

\newcommand{\etal}{et~al.\ }

\newcommand{\chii}[1]{\chi_\mathcal{I}(#1)}

\numberwithin{equation}{section}

\newcommand{\D}{\mathrm{deg}}
\newcommand{\dist}{\mathrm{dist}}

\usepackage{lineno}
\begin{document}

\title{\LARGE Hypergraph incidence coloring
\thanks{This work was supported by the National Natural Science Foundation of China (11631014) and the National Key Research \& Development Program of China (2017YFC0908405). }}

\author{Weichan Liu$^{1,2}$  \quad Guiying Yan$^{1,2}$\\
{\small  1. Academy of Mathematics and Systems Science, Chinese Academy of Sciences, Beijing, China}\\
{\small  2. University of Chinese Academy of Sciences, Beijing, China}\\
{\small  emails: wcliu@amss.ac.cn, yangy@amt.ac.cn}}
 \date{\today}

\maketitle

\begin{abstract}\baselineskip 0.60cm
An incidence of a hypergraph $\mathcal{H}=(X,S)$ is a pair $(x,s)$ with $x\in X$, $s\in S$ and $x\in s$. Two incidences $(x,s)$ and $(x',s')$ are adjacent if (i) $x=x'$, or (ii) $\{x,x'\}\subseteq s$ or $\{x,x'\}\subseteq s'$. 
A proper incidence $k$-coloring of a hypergraph $\mathcal{H}$ is a mapping $\varphi$ from 
the set of incidences of $\mathcal{H}$ to $\{1,2,\ldots,k\}$
so that $\varphi(x,s)\neq \varphi(x',s')$ for any two adjacent incidences $(x,s)$ and $(x',s')$ of $\mathcal{H}$.
The incidence chromatic number $\chi_I(\mathcal{H})$ of $\mathcal{H}$ is the minimum integer $k$ such that  $\mathcal{H}$ has a proper incidence $k$-coloring. In this paper we prove $\chi_I(\mathcal{H})\leq (4/3+o(1))r(\mathcal{H})\Delta(\mathcal{H})$ for every $t$-quasi-linear hypergraph with $t<<r(\mathcal{H})$ and sufficiently large 
$\Delta(\mathcal{H})$, where $r(\mathcal{H})$ is the maximum of the cardinalities of the edges in $\mathcal{H}$.
It is also proved that $\chi_I(\mathcal{H})\leq \Delta(\mathcal{H})+r(\mathcal{H})-1$ if  $\mathcal{H}$ is an $\alpha$-acyclic linear hypergraph, and this bound is sharp.

\end{abstract}

\baselineskip 0.60cm

\section{Introduction}

Let $\mathcal{H}$ be a hypergraph $(X,S)$, where $X$ is a vertex set and $S$ is an edge set which is a family of non-empty subsets of $X$. A hypergraph $\mathcal{H}'=(X',S')$ is a \textit{subhypergraph} of $\mathcal{H}$, written as $\mathcal{H}'\subseteq \mathcal{H}$, if $X'\subseteq X$ and $S'\subseteq S$. A hypergraph $(X',S')$ is an \textit{induced subhypergraph} of $\mathcal{H}$ on a set $Y\subseteq X$ of vertices, denoted by $\mathcal{H}[Y]$, if
$X'=Y$ and $S'=\{e~|~e\cap Y\neq \emptyset, e\in S\}$.

An \textit{$(x_1,x_k)$-path} of $\mathcal{H}$ is a sequence $(x_1,s_1,\ldots,x_{k-1},s_{k-1},x_k)$ 
with distinct vertices $x_1,\ldots,x_k$ and distinct edges $s_1,\ldots,s_{k-1}$ such that 
$\{x_i,x_{i+1}\}\subseteq  s_i$ for each $1\leq i\leq k-1$. 
Now 
$\mathcal{H}$ is \textit{connected} if there exists an $(x,x')$-path
for every two distinct vertices $x,x'\in X$, is \textit{$k$-uniform} if $|s|=k$ for every edge $s\in S$, and is \textit{linear} if $|s\cap s'|\leq 1$ for every two distinct edges $s,s'\in S$. 
We denote by $\D_{\mathcal{H}}(x)=|\{s~|~x\in s\in S\}|$ the \textit{degree} of $x$ in $\mathcal{H}$.  
Now $\mathcal{H}$ is \textit{$d$-regular} if $\D_{\mathcal{H}}(x)=d$ for every vertex $x\in X$. 
Let $\Delta(\mathcal{H})=\max\{\D_{\mathcal{H}}(x)~|~x\in X\}$, $\delta(\mathcal{H})=\min\{\D_{\mathcal{H}}(x)~|~x\in X\}$, and $r(\mathcal{H})=\max\{|s|~|~s\in S\}$. 

We call $(x,s)$ an \textit{incidence} of $\mathcal{H}$ if $x\in X$, $s\in S$ and $x\in s$. Let $I(\mathcal{H})$ be the set of incidences of $\mathcal{H}$.
Two incidences $(x,s)$ and $(x',s')$ are \textit{adjacent} if (i) $x=x'$, or (ii) $\{x,x'\}\subseteq s$ or $\{x,x'\}\subseteq s'$.

A \textit{proper incidence $k$-coloring} of $\mathcal{H}$ is a mapping $\varphi$: $I(\mathcal{H})\rightarrow \{1,2,\ldots,k\}$ so that $\varphi(x,s)\neq \varphi(x',s')$ for any two adjacent incidences $(x,s)$ and $(x',s')$ of $\mathcal{H}$. The \textit{incidence chromatic number} $\chi_I(\mathcal{H})$ of $\mathcal{H}$ is
the minimum integer $k$ such that $\mathcal{H}$ has a proper incidence $k$-coloring.
This notion generalizes the incidence chromatic number of graphs, which was introduced in 1993 by Brualdi and Quinn Massey \cite{Brualdi199351}.

In the literature, many topics concerning the incidence coloring of graphs were investigated, 
including 
incidence coloring of certain graph classes \cite{MAYDANSKIY2005131,Gregor2017174,Bonamy201729,Kardos2020345,Gregor201693,Bi20211,Dolama2004121,Dolama2005,Bonamy2014190,Shiu20086575}, incidence choosability \cite{Benmedjdoub201940}, interval incidence coloring \cite{Janczewski2014131}, fractional incidence coloring \cite{Yang}, incidence coloring game \cite{Andres20091980}, digraph incidence coloring \cite{Duffy2019191}, the complexity of the incidence coloring \cite{Li20081334}, and the application of incidence coloring to multi-frequency assignment problems \cite{Bi20211}. To our knowledge, there is no publication concerning the incidence coloring of hypergraphs.

In a graph $G=(V,E)$, the \textit{distance} of two vertices $u,v\in V$ in $G$ is the length (the number of edges) of the shortest path between $u$ and $v$, denoted by  $\dist_G(u,v)$.
The \textit{square} of $G$ is a graph $G^2$ with vertex set $V(G^2)=V$ and edge set $E(G^2)=E\cup \{uv~|~\dist_G(u,v)=2\}$.
The \textit{line graph} of $G$ is a graph $L(G)$ 
with vertex set $V(L(G))=\{v_e~|~e\in E\}$ and edge set $E(L(G))=\{v_ev_{e'}~|~\emptyset\neq e\cap e'\subseteq E,~e\neq e'\}$; then let $\dist_{G}(e,e')=\dist_{L(G)}(v_e,v_{e'})$ for two edges $e,e'\in E$.
A \textit{strong edge-$k$-coloring} of $G$ is a mapping $\phi:E(G)\rightarrow \{1,2,\ldots,k\}$ so that $\phi(e_1)\neq \phi(e_2)$ if $\dist_G(e_1,e_2)\leq 2$. The \textit{strong chromatic index} $\chi'_s(G)$ of $G$ is 
the minimum integer $k$ such that $G$ has a strong edge $k$-coloring.
Since a strong edge $k$-coloring of $G$ can be easily translated into a proper $k$-coloring of $L(G)^2$ and vise versa, 
  $\chi'_s(G)=\chi(L(G)^2)$
for every graph $G$.

The \textit{Levi graph} of $\mathcal{H}$ is a bipartite graph $B(\mathcal{H})=(V_1,V_2; E)$ where $V_1=X,V_2=S$ and $E=\{xs~|~x\in X,s\in S,x\in s\}$.
According to the definitions of Levi graph, proper incidence coloring, and strong edge coloring, the following is straightforward.

\begin{observation}\label{lem:bijection}
A hypergraph $\mathcal{H}$ has a proper incidence $k$-coloring if and only if 
$\mathcal{B}(\mathcal{H})$ has a strong edge $k$-coloring.
\end{observation}

To see this, we let $\varphi$ be an incidence coloring of $\mathcal{H}=(X,S)$ and let $\phi$ be an edge coloring of $\mathcal{B}(\mathcal{H})$ 
such that $\varphi(x,s)=\phi(xs)$ for every $x\in X$ and $s\in S$ with $x\in s$. Note that $(x,s)$ is an incidence of $\mathcal{H}$ and $xs$ is an edge of $\mathcal{B}(\mathcal{H})$.
Now $\phi$ is a strong edge coloring if and only if 
$\varphi(x_1,s_1)=\phi(x_1s_1)\neq \phi(x_2s_2)=\varphi(x_2,s_2)$ ($x_i\in X$, $s_i\in S$, $i=1,2$) whenever 
$1\leq \dist_{\mathcal{B}(\mathcal{H})}(x_1s_1,x_2s_2)\leq 2$, which is equivalent to say that one of the following holds: ($a$) $x_1=x_2$ and $s_1\neq s_2$, ($b$) $x_1\not=x_2$ and $s_1= s_2$, or ($c$) $x_1\not=x_2$, $s_1\not= s_2$ and either $x_1s_2$ or $x_2s_1$ is an edge of $\mathcal{B}(\mathcal{H})$, and thus that $(x_1,s_1)$ and $(x_2,s_2)$ are two adjacent incidences of $\mathcal{H}$. Hence $\phi$ is a strong edge coloring if and only if $\varphi$ is a proper incidence coloring.

Observation \ref{lem:bijection} immediately implies 
\begin{align}\label{eq:relationship}
  \chii{\mathcal{H}}=\chi'_s(\mathcal{B}(\mathcal{H}))=\chi(L(\mathcal{B}(\mathcal{H}))^2).  
\end{align}

We combine the first equality of \eqref{eq:relationship} with  known results on the strong chromatic index of graphs
to estimate upper bounds for $\chii{\mathcal{H}}$.
This story began in 1985, in which year Erd\H{o}s and Ne\v{s}et\v{r}il conjectured $\chi'_s(G)\leq  1.25\Delta(G)^2$ for every graph $G$  (note that this conjecture was first presented in a formal paper by Faudree \etal\cite{Faudree198983}). They also pointed out that the upper bound in the conjecture is sharp. However, this conjecture is far away to be completely resolved until now.

Let $G$ be a graph with sufficiently large maximum degree $\Delta$ .
As $2\Delta^2-2\Delta+1$ is a trivial greedy upper bound for $\chi'_s(G)$, finding the biggest constant $\varepsilon$ such that
$s\chi'(G)\leq (2-\varepsilon)\Delta^2$ is  interesting. 
The extremal graph of Erd\H{o}s and Ne\v{s}et\v{r}il \cite{Faudree198983} implied $\varepsilon\leq 0.75$. In 1997, Molloy and Reed \cite{Molloy1997103} proved that $\varepsilon\geq 0.002$, and thus $\chi'_s(G)\leq  1.998\Delta^2$.
This was the first breakthrough result beating the trivial greedy upper bound for $\chi'_s(G)$.
The next improvement was due to Bruhn and Joos \cite{Bruhn201821}, who proved in 2018 that $\varepsilon\geq 0.07$.
Soon after, Bonamy, Perrett, and Postle \cite{bonamy2018colouring} improved this by showing that $\varepsilon\geq 0.165$.
Very recently (actually in 2021), Hurley, de Verclos, and Kang \cite{Hurley2021135} came to the best known result that $\varepsilon\geq 0.228$. This implies $\chi'_s(G)\leq  1.772\Delta^2$.

Let $\mathcal{H}$ be a hypergraph and denote $\mathcal{B}(\mathcal{H})=(\widetilde{V},\widetilde{E})$.
Since $\Delta(\mathcal{B}(\mathcal{H}))=\varrho(\mathcal{H})=\max\{r(\mathcal{H}),\Delta(\mathcal{H})\}$, we combine the first equality of \eqref{eq:relationship} with the result of Hurley, de Verclos, and Kang to obtain 
\begin{align}\label{eqqqq1}
\chii{\mathcal{H}}\leq 1.772\varrho(\mathcal{H})^2
\end{align}
for every hypergraph $\mathcal{H}$ with sufficiently large $\varrho(\mathcal{H})$.
On the other hand, since $\big|\{e'\in \widetilde{E}~|~1\leq \dist_{\mathcal{B}(\mathcal{H})}(e,e')\leq 2\}\big|\leq (\Delta(\mathcal{H})-1)r(\mathcal{H})+(r(\mathcal{H})-1)\Delta(\mathcal{H})=2r(\mathcal{H})\Delta(\mathcal{H})-r(\mathcal{H})-\Delta(\mathcal{H})<2r(H)\Delta(\mathcal{H})$ for each $e\in \widetilde{E}$,  
$\chi(L(\mathcal{B}(\mathcal{H}))^2)\leq \Delta(L(\mathcal{B}(\mathcal{H}))^2)+1\leq 2r(H)\Delta(\mathcal{H})$, and thus 
\begin{align}\label{eqqqq2}
   \chii{\mathcal{H}}\leq 2r(H)\Delta(\mathcal{H}).
\end{align}
for every hypergraph $\mathcal{H}$ by the second  equality of \eqref{eq:relationship}.

The first goal of this paper is to break the 1.772 barrier of \eqref{eqqqq1} or the 2 barrier of \eqref{eqqqq2}.
This seems extremely challenging yet alternatively it is natural to consider the same problem for restricted hypergraphs.
Maybe the first one special hypergraph class in our mind is the class of  
linear hypergraphs, which are well studied in the literature \cite{Henning}.
To go a step further, we introduce a new notion generalizing linear hypergraphs.

A hypergraph $\mathcal{H}=(X,S)$ is a \textit{$\mathbf{\textit{t}}$-quasi-linear hypergraph} if
\begin{itemize}
    \item $|s\cap s'|\leq t$ for every two distinct edges $s,s'\in S$;
    \item $\big|\{s\in S~|~\{x,x'\}\subseteq s\}\big|\leq t$ for every two distinct vertices $x,x'\in X$.
\end{itemize}

\noindent Clearly, linear hypergraphs are exactly $1$-quasi-linear hypergraphs.

In the next sections we devote to proving the following theorem. A more detailed form of the result will be presented in Section \ref{sec:4} by Theorem \ref{maintheorem}.

\begin{thm}\label{thm:limits}
Let $\mathcal{H}$ be a $t$-quasi-linear hypergraph with $t<<r(\mathcal{H})$. If $\Delta(\mathcal{H})$ is sufficiently large, then
\begin{align*}
    \chii{\mathcal{H}}\leq \bigg(\frac{4}{3}+o(1)\bigg)r(\mathcal{H})\Delta(\mathcal{H}).
\end{align*}
\end{thm}

Another special hypergraph class we investigate in this paper is the class of $\alpha$-acyclic hypergraphs. 
Note that the $\alpha$-acyclicity is an important notion in database theory \cite{zbMATH05964612}.
Many NP-hard problems concerning databases can be solved in polynomial time when restricted to
instances for which the corresponding hypergraphs are $\alpha$-acyclic \cite{zbMATH04014074,zbMATH03868639}.

Specifically, Graham \cite{Gra79}, and independently, Yu and \"Oszoyoglu \cite{Yu1979306} defined the $\alpha$-acyclicity of a hypergraph
through a property of reducibility to the empty hypergraph via a certain ``reduction'' process called \textit{GYO-reduction}.
Given a hypergraph $\mathcal{H}$, the GYO-reduction applies the following operations repeatedly to $\mathcal{H}$ until none can be applied anymore: (i) eliminate	a vertex in	only one edge; (ii) eliminate an edge	contained in another; and (iii)
eliminate	an edge containing no vertex.
A hypergraph is \textit{$\alpha$-acyclic} if GYO-reduction on it results in an empty hypergraph. 

An interesting phenomenon pointed out by Simon and Wojtczak \cite{Simon2017402} for the $\alpha$-acyclicity is that a subhypergraph of an $\alpha$-acyclic hypergraph may not be $\alpha$-acyclic. In other words, $\alpha$-acyclic is not a hereditary property for hypergraphs.
Surprisingly, $\alpha$-acyclic is definitely a hereditary property for linear hypergraphs. We will show the reason for this in Section \ref{sec:6} and then apply this property to prove the following.

\begin{thm}\label{cort}
If $\mathcal{H}$ is an $\alpha$-acyclic linear hypergraph, then $\chii{\mathcal{H}}\leq \Delta(\mathcal{H})+r(\mathcal{H})-1$. In particular, if we further assume $\mathcal{H}$ is $k$-uniform, then $\chii{\mathcal{H}}= \Delta(\mathcal{H})+k-1$.
\end{thm}

\section{Properties of quasi-linear hypergraphs} 

A bipartite graph $G$ with bipartition $A$ and $B$ is \textit{$(a,b)$-bipartite} if $\Delta(A)=a$ and $\Delta(B)=b$, and is \textit{$(a,b)$-regular} if each vertex of $A$ has degree $a$ and each vertex of $B$ has degree $b$. A graph $G$ is \textit{$H$-free} if $G$ does not contain $H$ as a subgraph.

\begin{prop}\label{prop:regular-bipartite}
If $\mathcal{H}$ is a $k$-uniform $\Delta(\mathcal{H})$-regular $t$-quasi-linear hypergraph, then $\mathcal{B}(\mathcal{H})$ is a $K_{2,t+1}$-free $(k,\Delta(\mathcal{H}))$-regular bipartite graph.
\end{prop}

\begin{proof}
Let $\mathcal{H}=(X,S)$ and $\Delta=\Delta(\mathcal{H})$.
Since $\mathcal{H}$ is $k$-uniform and $\Delta$-regular,
$B:=\mathcal{B}(\mathcal{H})$ is a $(k,\Delta)$-regular bipartite graph with bipartition $X$ and $S$, where 
every vertex of $X$ has degree $\Delta$ and every vertex of $S$ has degree $k$ in $B$. 

Suppose for a contradiction that $B$ contains a copy of $K_{2,t+1}$ as a subgraph. We distinguish two asymmetric cases.
If there are two vertices $x_1,x_2\in X$ and $t+1$ vertices $s_1,\ldots,s_{t+1}\in S$ of $B$ such that $x_is_j\in E(B)$ (thus $x_i\in s_j\in S$ in $\mathcal{H}$) for each $1\leq i\leq 2$ and $1\leq j\leq t+1$, then $\{x_1,x_2\}\subseteq \bigcap_{j=1}^{t+1}s_j$.
If there are two vertices $s_1,s_2\in S$ and $t+1$ vertices $x_1,\ldots,x_{t+1}\in X$ of $B$ such that $x_is_j\in E(B)$ (thus $x_i\in s_j\in S$ in $\mathcal{H}$) for each $1\leq i\leq t+1$ and $1\leq j\leq 2$, then 
$\{x_1,\ldots,x_{t+1}\}\subseteq s_1\cap s_2$.
Each of the above two conclusions contradicts the definition of the $t$-quasi-linearity.
\end{proof}

\begin{prop}\label{prop:supergraph}
If $\mathcal{H}$ is a $t$-quasi-linear hypergraph, 
then there exists an $r(\mathcal{H})$-uniform $\Delta(\mathcal{H})$-regular $t$-quasi-linear hypergraph $\mathcal{H}^*$ containing $\mathcal{H}$ as a subhypergraph. 
\end{prop}

\begin{proof}
Let $k=r(\mathcal{H})$ and $\Delta=\Delta(\mathcal{H})$.
We construct the desired hypergraph by following two steps.

Step 1. \textit{Construct a $k$-uniform $t$-quasi-linear hypergraph $\mathcal{H}'$ such that $\Delta(\mathcal{H}')=\Delta$ and $\mathcal{H}\subseteq \mathcal{H}'$.}

Let $\mathcal{H}=(X,S)$ and let $S_0=\{s\in S~|~|s|<k\}$. 
If $S_0=\emptyset$, then let $\mathcal{H}'=\mathcal{H}$; otherwise let $S_0=\{s_1,\ldots,s_\ell\}$ ($\ell\geq 1$).
Let $d_i=k-|s_i|$ ($i=1,\ldots,\ell$). Into each $s_i$ with $1\leq i\leq \ell$, we import a set $z_i$ of $d_i$ new vertices
in such a way that $z_i\cap X=\emptyset$ ($1\leq i\leq \ell$) and $z_i\cap z_j=\emptyset$ ($1\leq i\neq j\leq \ell$).
Now let $\mathcal{H}'=(X',S')$, where $X'=X\cup \bigcup_{i=1}^\ell z_i$ and $S'=(S\setminus S_0) \cup \bigcup_{i=1}^\ell (s_i\cup z_i)$.
One can easily check that $\mathcal{H}'$ is $k$-uniform and $t$-quasi-linear, $\Delta(\mathcal{H}')=\Delta(\mathcal{H})=\Delta$, and $\mathcal{H}\subseteq \mathcal{H}'$.

Let $\delta=\delta(\mathcal{H}')$. If $\delta=\Delta$, then $\mathcal{H}'$ is $\Delta$-regular and let $\mathcal{H}^*=\mathcal{H}'$, as desired; otherwise we turn to Step 2. 

Step 2. \textit{Construct a $k$-uniform $t$-quasi-linear hypergraph $\mathcal{H}''$ such that $\Delta(\mathcal{H}'')=\Delta$, $\mathcal{H}'\subseteq \mathcal{H}''$, and $\delta(\mathcal{H}'')>\delta(\mathcal{H}')$.}

Let $Y'=\{x\in X'~|~\D_{\mathcal{H}'}(x)<\Delta\}$.
As we are in this step, $\delta<\Delta$ and thus $Y'\neq \emptyset$.
So we may assume $Y'=\{x_1,\ldots,x_p\}$ ($p\geq 1$) and denote $X'=Y'\cup \{x_{p+1},\ldots,x_n\}$. 
Make $k$ vertex-disjoint copies $\mathcal{H'}_1,\ldots, \mathcal{H'}_k$ of $\mathcal{H'}$ and denote $\mathcal{H}'_i=(X'_i,S'_i)$ ($1\leq i\leq k$), where $X'_i=\bigcup_{j=1}^n \{x^i_j\}$ is the vertex set of each copy $\mathcal{H'}_i$ homomorphic to the vertex set $X'=\bigcup_{j=1}^n \{x_j\}$ of the original hypergraph $\mathcal{H}'$.
Now let $\mathcal{H''}=(X'',S'')$, where $X''=\bigcup_{i=1}^k X'_i$ and $S''=\bigcup_{i=1}^k S'_i \cup \bigcup_{j=1}^p\{x_j^1,\ldots,x^k_j\}$. One can immediately check that $\mathcal{H}''$ is $k$-uniform and $t$-quasi-linear and $\mathcal{H}'\subseteq \mathcal{H}''$. Moreover, $\Delta(\mathcal{H}'')=\Delta(\mathcal{H}')=\Delta$ and $\delta(\mathcal{H}'')=\delta(\mathcal{H}')+1$, as desired.

If $\delta(\mathcal{H}'')=\Delta$, then $\mathcal{H}''$ is $\Delta$-regular and let $\mathcal{H}^*=\mathcal{H}'$, as desired; otherwise we return back to Step 2 by letting $\mathcal{H}':=\mathcal{H}''$. This iteration would stop with a desired hypergraph $\mathcal{H}^*$ after we have visited Step 2 $\Delta-\delta$ times.
\end{proof}

The idea of proving Theorem \ref{thm:limits} is to bound $\chii{\mathcal{H}}$ above by $\chii{\mathcal{H}^*}$ as $\mathcal{H} \subseteq \mathcal{H}^*$, where 
$\mathcal{H}^*$ is the hypergraph obtained from $\mathcal{H}$ by Proposition \ref{prop:supergraph}.
Since $\Delta(\mathcal{H}^*)=\Delta(\mathcal{H})$, $B(\mathcal{H}^*)$ is a $K_{2,t+1}$-free $(r(\mathcal{H}),\Delta(\mathcal{H}))$-regular bipartite graph by Proposition \ref{prop:regular-bipartite}. Since
$\chii{\mathcal{H}^*}=\chi'_s(B(\mathcal{H}^*))$ by \eqref{eq:relationship}, it is sufficient to prove an upper bound for the strong chromatic index of birugular bipartite graphs, or exactly, $(r(\mathcal{H}),\Delta(\mathcal{H}))$-regular bipartite graphs.

\section{The sparsity of biregular bipartite graphs}

In a graph $G=(V,E)$, we let $N_G(v)=\{u\in V~|~uv\in E\}$ and $E_G(v)=\{e\in E~|~v\in e\}$.
For each edge $e\in E$, let 
\[D_2(e)=\{e'\in E\setminus\{e\}~|~\dist_G(e,e')\leq 2\}.\]
For every edge $e'\in D_2(e)$, let
\[\zeta(e',e)=|D_2(e')\cap D_2(e)|.\]

\begin{lem}\label{lem:counting-edge}
If $G$ is a $K_{2,t+1}$-free $(a,b)$-regular bipartite graph $(U,V; E)$, then for each edge $e\in E$,
\begin{align*}
   \sum_{e'\in D_2(e)}\zeta(e',e)\leq (4a-3)tb^2+o(b^2).
\end{align*}
\end{lem}

\begin{proof}
Assume by symmetry that $\D_G(u)=a$ for every $u\in U$ and $\D_G(v)=b$ for every $v\in V$. Denote $e$ by $uv$ ($u\in U$ and $v\in V$).
Let $N_G(v)\setminus \{u\}=\{u_1,u_2,\ldots,u_{b-1}\}$ and $N_G(u)\setminus \{v\}=\{v_1,v_2,\ldots,v_{a-1}\}$.
Denote $u_0=u$ and $v_0=v$.
Since $G$ is bipartite, $(N_G(v)\setminus \{u\}) \cap (N_G(u)\setminus \{v\})=\emptyset$. 
Let
\begin{align*}
    S_1&=\{vu_i~|~1\leq i\leq b-1\},\\
    S_2&=\{uv_i~|~1\leq i\leq a-1\},\\
    S_3&=\{u_iw~|~1\leq i\leq b-1, w\in N_G(u_i)\setminus \{v\}\}, and\\
    S_4&=\{v_iw~|~1\leq i\leq a-1, w\in N_G(v_i)\setminus \{u\}\}.
\end{align*}
Clearly, 
\begin{align}\label{4partition}
    D_2(e)= S_1\cup S_2\cup S_3\cup S_4
\end{align}

For any edge $e'=vu_i\in S_1$ ($1\leq i\leq b-1$),
\begin{align}\label{case1.1}
D_2(e')\cap (S_1\cup S_2)=(S_1\cup S_2)\setminus \{e'\} ~{\rm and}~ D_2(e')\cap S_3=S_3
\end{align}
If there is one edge $e''\in D_2(e')\cap (S_4\setminus S_3)$, then $e''=v_jw$ for some $1\leq j\leq a-1$ and $w\neq u$, and we further have
$v_j\in N_G(u_i)$ and $E_G(v_j)\setminus \{uv_j,u_iv_j\}\subseteq D_2(e')\cap (S_4\setminus S_3)$.
This implies 
\begin{align}\label{case1.2}
    D_2(e')\cap (S_4\setminus S_3)= \bigcup_{v_j\in N_G(u_i)} E_G(v_j)\setminus \{uv_j,u_iv_j\}.
\end{align}
Since $G$ is $K_{2,t+1}$-free, 
\begin{align}\label{free1}
   |N_G(u_i)\cap \{v_1,\ldots,v_{a-1}\}|\leq t-1. 
\end{align}
 Hence by 
\eqref{4partition}, \eqref{case1.1}, and \eqref{case1.2}, we conclude 
\begin{align*}
    D_2(e')\cap D_2(e)&\subseteq (D_2(e')\cap (S_1\cup S_2)) \cup (D_2(e')\cap S_3) \cup  (D_2(e')\cap (S_4\setminus S_3) \\
    &=(S_1\cup S_2\cup S_3)\setminus \{e'\} \cup  \bigcup_{v_j\in N_G(u_i)} E_G(v_j)\setminus \{uv_j,u_iv_j\}.
\end{align*}
Thus 
\begin{align}
   \notag\zeta(e',e)&\leq  |S_1\cup S_2\cup S_3|-1+(t-1)(b-2)\\
   \notag&\leq (b-1)+(a-1)+(b-1)(a-1)-1+(t-1)(b-2)\\
   \notag&=(a+t-1)b-2t.
\end{align}
by \eqref{free1}, and
\begin{align}\label{cal-1}
    \sum_{e'\in S_1}\zeta(e',e)\leq (b-1)\big((a+t-1)b-2t\big)
\end{align}
and by symmetry we further have
\begin{align}\label{cal-2}
    \sum_{e'\in S_2}\zeta(e',e)\leq (a-1)\big((b+t-1)a-2t\big).
\end{align}

For any edge $e'=u_iw\in S_3$ ($1\leq i\leq b-1$, $w\in N_G(u_i)\setminus \{v\}$),
\begin{align}\label{case2.1}
D_2(e')\cap S_1=S_1.
\end{align}
Let $A=N_G(v)\cap N_G(w)$ and $\tilde{A}=N_G(v)\setminus A$. Note that $u_i\in A$ and thus $u_i\not\in \tilde{A}$.
Now
\begin{align}\label{3.8}
    \bigg(\bigcup_{u_j\in A} E_G(u_j)\setminus \{vu_j\}\bigg)\setminus \{e'\}\subseteq D_2(e').
\end{align}
Since $G$ is $K_{2,t+1}$-free, 
$|A|\leq t$.
Hence 
\begin{align}\label{case2.3}
    \bigg|D_2(e') \cap \bigg(\bigcup_{u_j\in A} E_G(u_j)\setminus\{vu_j\}\bigg)\bigg|\leq \bigg|\bigcup_{u_j\in A} E_G(u_j)\setminus\{vu_j\}\bigg|-1\leq t(a-1)-1
\end{align}
by \eqref{3.8}.

If there is one edge $e''\in D_2(e')\cap (S_2\cup S_3)$ incidence with some vertex in $\tilde{A}$, then $e''=u_jw'$ for some $u_j\in \tilde{A}$ ($0\leq j\leq b-1$) and $w'\neq v,w$, and we further have $u_iw'\in E$, i.e., $w'\in N_G(u_i)\setminus \{v,w\}$. 
Hence
\begin{align}
    D_2(e')\cap \bigg(\bigcup_{u_j\in \tilde{A}} E_G(u_j)\setminus\{vu_j\}\bigg)
    \notag &=\big\{u_jw'~|~u_j\in\tilde{A}, w'\in N_G(u_i)\setminus\{v,w\}\big\}\\
   \label{case2.2} &=\big\{u_jw'~|~u_j\in N_G(v)\cap N_G(w')\cap \tilde{A}, w'\in N_G(u_i)\setminus\{v,w\}\big\}
\end{align}
Since $G$ is $K_{2,t+1}$-free, $|N_G(v)\cap N_G(w')\cap \tilde{A}|\leq t-1$ for every $w'\in N_G(u_i)\setminus\{v,w\}$ (recall that $u_i\not\in \tilde{A}$). Hence 
\begin{align}\label{case2.4}
    \bigg|D_2(e') \cap \bigg(\bigcup_{u_j\in \tilde{A}} E_G(u_j)\setminus\{vu_j\}\bigg)\bigg|\leq (t-1)(a-2)
\end{align}
by \eqref{case2.2}.

Combining \eqref{case2.3} and \eqref{case2.4}, we conclude
\begin{align}\label{S2S3}
   \notag |D_2(e')\cap (S_2\cup S_3)|&=\bigg|D_2(e') \cap \bigg(\bigcup_{u_j\in A} E_G(u_j)\setminus\{vu_j\}\bigg)\bigg|+\bigg|D_2(e') \cap \bigg(\bigcup_{u_j\in \tilde{A}} E_G(u_j)\setminus\{vu_j\}\bigg)\bigg|\\
    &\leq (2a-3)t-a+1.
\end{align}

Let $B=N_G(u)\cap N_G(u_i)\setminus \{v\}$ and $\tilde{B}=N_G(u)\setminus (B\cup \{v\})$.
Now
\begin{align}\label{3.13}
    \bigcup_{v_j\in B} E_G(v_j)\setminus \{uv_j,u_iv_j\} \subseteq D_2(e')\setminus (S_2\cup S_3).
\end{align}
Since $G$ is $K_{2,t+1}$-free, 
$|B|\leq t-1$.
Hence 
\begin{align}
\notag \bigg|D_2(e')\cap S_4\cap  \bigg(\bigcup_{v_j\in B} E_G(v_j)\setminus \{uv_j\}\bigg)\bigg|
    &\leq \bigg|D_2(e') \cap \bigg(\bigcup_{v_j\in B} E_G(v_j)\setminus \{uv_j,u_iv_j\}\bigg)\bigg|\\
    \label{case2.3}&= \bigg|\bigcup_{v_j\in B} E_G(v_j)\setminus \{uv_j,u_iv_j\}\bigg|\leq (t-1)(b-2).
\end{align}
by \eqref{3.13}.

If there is one edge $e''\in D_2(e')\cap S_4$ incidence with some vertex in $\tilde{B}$, then $e''=v_jw'$ for some $v_j\in \tilde{B}$ ($1\leq j\leq a-1$) and $w'\neq u,u_i$, and we further have $ww'\in E$, i.e., $w'\in N_G(v_j)\cap N_G(w)$. 
Hence
\begin{align}
    D_2(e')\cap S_4\cap  \bigg(\bigcup_{v_j\in \tilde{B}} E_G(v_j)\setminus \{uv_j\}\bigg)
    \notag &=\big\{v_jw'~|~v_j\in\tilde{B}, w'\in N_G(v_j)\cap N_G(w)\big\}\\
   \label{case2.3} &=\big\{v_jw'~|~v_j\in N_G(u)\setminus (B\cup \{v\}), w'\in N_G(v_j)\cap N_G(w)\big\}
\end{align}
Since $G$ is $K_{2,t+1}$-free, $|N_G(v_j)\cap N_G(w)|\leq t$ for every $v_j\in N_G(u)\setminus (B\cup \{v\})$. Hence 
\begin{align}\label{case2.5}
    \bigg|D_2(e')\cap S_4\cap  \bigg(\bigcup_{v_j\in \tilde{B}} E_G(v_j)\setminus \{uv_j\}\bigg)\bigg|\leq t(a-1)
\end{align}
by \eqref{case2.3}.
According to \eqref{case2.4} and \eqref{case2.5}, we conclude
\begin{align}\label{S4}
   \notag |D_2(e')\cap S_4|&=\bigg| D_2(e')\cap S_4\cap  \bigg(\bigcup_{v_j\in B} E_G(v_j)\setminus \{uv_j\}\bigg)\bigg|+\bigg|D_2(e')\cap S_4\cap  \bigg(\bigcup_{v_j\in \tilde{B}} E_G(v_j)\setminus \{uv_j\}\bigg)\bigg|\\
    &\leq (a+b-3)t-b+2.
\end{align}

Combining \eqref{case2.1}, \eqref{S2S3}, and \eqref{S4} together, we obtain
\begin{align*}
    \zeta(e',e)=|D_2(e')\cap D_2(e)|&=|D_2(e')\cap S_1|+|D_2(e')\cap (S_2\cup S_3)|+|D_2(e')\cap S_4|\\
    &=|S_1|+|D_2(e')\cap (S_2\cup S_3)|+|D_2(e')\cap S_4|\\
    &\leq (b-1)+(2a-3)t-a+1+(a+b-3)t-b+2\\
    &=(3t-1)(a-2)+bt
\end{align*}
for every $e'\in S_3$ by \eqref{4partition}.
It follows
\begin{align}\label{cal-3}
    \sum_{e'\in S_3}\zeta(e',e)\leq (a-1)(b-1)\big((3t-1)(a-2)+bt\big)
\end{align}
and by symmetry we further have
\begin{align}\label{cal-4}
    \sum_{e'\in S_4}\zeta(e',e)\leq (a-1)(b-1)\big((3t-1)(b-2)+at\big).
\end{align}

Finally, we combine \eqref{cal-1}, \eqref{cal-2}, \eqref{cal-3}, and \eqref{cal-4} together and then obtain
\begin{align*}
   \sum_{e'\in D_2(e)}\zeta(e',e)&=\sum_{i=1}^4\sum_{e'\in S_i}\zeta(e',e)\\
   &=(b-1)\big((a+t-1)b-2t\big)+(a-1)\big((b+t-1)a-2t\big)\\
   &~~~~+(a-1)(b-1)\big((3t-1)(a-2)+bt\big)+(a-1)(b-1)\big((3t-1)(b-2)+at\big)\\
   &=(4k-3)tb^2+\bigg(4ta^2-(20t-4)a+13t-4\bigg)b-\bigg(3ta^2-(13t-4)a+8t-4\bigg)\\
   &=(4k-3)tb^2+o(b^2)
\end{align*}
as desired.
\end{proof}

A graph $G$ is \textit{$\sigma$-sparse} if for every vertex $v$ of $G$, the graph induced by $N_G(v)$ has at most $(1-\sigma)\binom{\Delta(G)}{2}$ edges.

\begin{thm}\label{lem:a}
Let $G$ be a $K_{2,t+1}$-free $(a,b)$-regular bipartite graph with $t<a$. 
For each number $\varepsilon_{\ref{lem:a}}>0$, there exists an integer $B_{\ref{lem:a}}$ such that if $b \geq B_{\ref{lem:a}}$ then 
$L(G)^2$ is 
\begin{align*}
   \bigg( 1-\frac{(4a-3)t}{(2a-1)^2}-\varepsilon_{\ref{lem:a}} \bigg)\texttt{-}sparse.
\end{align*}
\end{thm}

\begin{proof}
Since $G$ is an $(a,b)$-regular bipartite graph, $\Delta(L(G)^2)=(b-1)a+(a-1)b=2ab-a-b$.
For each vertex $v_e$ of $L(G)^2$, the graph induced by $N_{L(G)^2}(v_e)$ has at most $\frac{1}{2}f(a,b,t)$ edges by Lemma \ref{lem:counting-edge}.

Let 
\begin{align*}
  f(a,b,t)&=(4a-3)tb^2+o(b^2),\\
  g(a,b,t)&=2\binom{2ab-a-b}{2}=(2a-1)^2b^2-(4a^2-1)b+a(a+1).
\end{align*}
Observe $g(a,b,t)=(2a-1)^2b^2+o(b^2)$, so
\begin{align*}
 \lim\limits_{b\to+\infty} \frac{f(a,b,t)}{g(a,b,t)}=\frac{(4a-3)tb^2+o(b^2)}{(2a-1)^2b^2+o(b^2)}=\frac{(4a-3)t}{(2a-1)^2},
\end{align*}

For any small number $\varepsilon_{\ref{lem:a}}>0$, 
there exists an integer $B_{\ref{lem:a}}$ such that if $b\geq B_{\ref{lem:a}}$ then
\begin{align*}
    \bigg|\frac{f(a,b,t)}{g(a,b,t)}-\frac{(4a-3)t}{(2a-1)^2}\bigg|<\varepsilon_{\ref{lem:a}}.
\end{align*}
This follows
\begin{align*}\label{eq-bound}
    \frac{1}{2}f(a,b,t)&<\frac{1}{2}\bigg(\frac{(4a-3)t}{(2a-1)^2}+\varepsilon_{\ref{lem:a}}\bigg)g(a,b,t)\\
    &=
    \bigg(\frac{(4a-3)t}{(2a-1)^2}+\varepsilon_{\ref{lem:a}}\bigg)\binom{2ab-a-b}{2}\\
    &=\bigg(\frac{(4a-3)t}{(2a-1)^2}+\varepsilon_{\ref{lem:a}}\bigg)\binom{\Delta(L(G)^2)}{2}
\end{align*}
and thus
$L(G)^2$ is 
   $\bigg( 1-\frac{(4a-3)t}{(2a-1)^2}-\varepsilon_{\ref{lem:a}} \bigg)$-sparse.
\end{proof}

\noindent \textbf{Remark:} For a $K_{2,t+1}$-free $(a,b)$-regular bipartite graph $G$, we would assume $t<a$ or $t<b$, for otherwise $G$ does not contain a subgraph isomorphic to $K_{2,t+1}$ and thus the condition of $K_{2,t+1}$-free would be vacuous, and what is worse, $G$ is possible to be the complete bipartite graph $K_{a,b}$ and then $L(G)^2$ cannot be $\sigma$-sparse for any $\sigma<1$. This is indeed the reason why we assume $t<a$ in the statement of Lemma \ref{lem:a}.

\section{Proof of Theorem \ref{thm:limits}}\label{sec:4}

In this section we complete the proof of Theorem \ref{thm:limits}.

\begin{lem}\label{lem:newbound}\cite{Hurley2021135}
For each  $\varepsilon_{\ref{lem:newbound}}>0$ and $0\leq \sigma\leq 1$, there exists an integer $B_{\ref{lem:newbound}}$ such that
\[\chi(G)\leq (1-\sigma/2+\sigma^{3/2}/6+\varepsilon_{\ref{lem:newbound}})\Delta(G)\] for any $\sigma$-sparse graph $G$ with $\Delta(G)\geq B_{\ref{lem:newbound}}$.
\end{lem}

\begin{lem}\label{thm:main-1}
Let $G$ be a $K_{2,t+1}$-free $(a,b)$-regular bipartite graph with $t<a$.
For each  $\varepsilon_{\ref{thm:main-1}}>0$, there exists an integer $B_{\ref{thm:main-1}}$ such that
if  $b\geq B_{\ref{thm:main-1}}$ then
\[\chi(L(G)^2)\leq (Z(a,t)+\varepsilon_{\ref{thm:main-1}})\Delta(L(G)^2),\]
where 
\begin{align*}
    Z(a,t)=\frac{1}{2}\bigg(1+\frac{(4a-3)t}{(2a-1)^2}\bigg)+\frac{1}{6}\bigg(1-\frac{(4a-3)t}{(2a-1)^2}\bigg)^{3/2}.
\end{align*}
\end{lem}

\begin{proof}
Let 
\begin{align*}
   \sigma & = 1-\frac{(4a-3)t}{(2a-1)^2}-\varepsilon_{\ref{thm:main-1}}
\end{align*}
Now $1-\sigma/2+\sigma^{3/2}/6\leq Z(a,t)+\varepsilon_{\ref{thm:main-1}}/2$.

Let $B_{\ref{lem:a}}$ and $B_{\ref{lem:newbound}}$ be the integers satisfying Lemmas \ref{lem:a} and \ref{lem:newbound} where we input $\varepsilon_{\ref{lem:a}}$ and $\varepsilon_{\ref{lem:newbound}}$ by $\varepsilon_{\ref{thm:main-1}}$ and $\frac{1}{2}\varepsilon_{\ref{thm:main-1}}$, respectively.
Let $B_{\ref{thm:main-1}}=\max\{B_{\ref{lem:a}},B_{\ref{lem:newbound}}+1\}$ and assume $b>B_{\ref{thm:main-1}}$.

Since $b>B_{\ref{thm:main-1}}\geq B_{\ref{lem:a}}$, $L(G)^2$ is $\sigma$-sparse by Lemma \ref{lem:a}.
Since $\Delta(L(G)^2)=2ab-a-b\geq b-1>B_{\ref{thm:main-1}}-1\geq B_{\ref{lem:newbound}}$,
\begin{align*}
    \chi(L(G)^2)&\leq (1-\sigma/2+\sigma^{3/2}/6+\varepsilon_{\ref{thm:main-1}}/2)\Delta(L(G)^2)\\
    &\leq (Z(a,t)+\varepsilon_{\ref{thm:main-1}}/2+\varepsilon_{\ref{thm:main-1}}/2)\Delta(L(G)^2)\\
&=(Z(a,t)+\varepsilon_{\ref{thm:main-1}})\Delta(L(G)^2)
\end{align*}
by Lemma \ref{lem:newbound}.
\end{proof}

We are now ready to complete the proof of Theorem \ref{thm:limits} by the following theorem.

\begin{thm}\label{maintheorem}
Let $\mathcal{H}$ be a $t$-quasi-linear hypergraph with $t<r(\mathcal{H})=k$. If $\Delta(\mathcal{H})$ is sufficiently large, then
\begin{align*}
    \chii{\mathcal{H}}\leq W(k,t)r(\mathcal{H})\Delta(\mathcal{H}),
\end{align*}
where
\begin{align*}
    W(k,t)=\bigg(1+\frac{(4k-3)t}{(2k-1)^2}\bigg)+\frac{1}{3}\bigg(1-\frac{(4k-3)t}{(2k-1)^2}\bigg)^{3/2}.
\end{align*}
In particular,
\begin{align*}
    \chii{\mathcal{H}}\leq \bigg(\frac{4}{3}+o(1)\bigg)r(\mathcal{H})\Delta(\mathcal{H}).
\end{align*}
if $t<<r(\mathcal{H})$.
\end{thm}

\begin{proof}
By Propositions \ref{prop:regular-bipartite} and \ref{prop:supergraph}, there exists a $t$-quasi-linear hypergraph $\mathcal{H}^*$ containing $\mathcal{H}$ as a subhypergraph such that 
$\mathcal{B}(\mathcal{H}^*)$ is a $K_{2,t+1}$-free $(k,\Delta(\mathcal{H}))$-regular bipartite graph.
Therefore,
\begin{align*}
\chii{\mathcal{H}}&\leq \chii{\mathcal{H}^*}=\chi(L(\mathcal{B}(\mathcal{H}^*))^2)\\
&\leq (Z(k,t)+o(1))\Delta(L(\mathcal{B}(\mathcal{H}^*))^2)\\
&=(Z(k,t)+o(1))(2k\Delta(\mathcal{H})-k-\Delta(\mathcal{H}))\\
&\leq 2Z(k,t)k\Delta(\mathcal{H})\\
&=W(k,t)r(\mathcal{H})\Delta(\mathcal{H})
\end{align*}
by \eqref{eq:relationship} and by Lemma \ref{thm:main-1}.

It is easy to check that \[W(k,t)\xrightarrow[t=o(k)]{k\rightarrow +\infty} \frac{4}{3}.\]
Hence if $t<<r(\mathcal{H})$, then 
\begin{align*}
    \chii{\mathcal{H}}\leq \bigg(\frac{4}{3}+o(1)\bigg)r(\mathcal{H})\Delta(\mathcal{H}),
\end{align*}
as desired.
\end{proof}

\section{Discussions on Theorem \ref{thm:limits}}

Since linear hypergraphs are exactly $1$-quasi linear hypergraphs, we deduce the following from 
Theorem \ref{maintheorem}.

\begin{cor}\label{cor:new}
Let $\mathcal{H}$ be a linear hypergraph. If $\Delta(\mathcal{H})$ is sufficiently large, then 
\begin{align*}
    \chii{\mathcal{H}}\leq f(r(\mathcal{H}))\Delta(\mathcal{H}),
\end{align*}
where
$f(r(\mathcal{H}))=W(r(\mathcal{H}),1)r(\mathcal{H})$.
\end{cor}

We naturally assume $r(\mathcal{H})\geq 3$ in Corollary \ref{cor:new}. 
Since $W(r(\mathcal{H}),1)$ is a deceasing function of $r(\mathcal{H})$,
$$f(r(\mathcal{H}))\leq W(3,1)r(\mathcal{H})\leq 1.531r(\mathcal{H}).$$
We write down this result as a corollary.

\begin{cor}\label{cor:new1}
Let $\mathcal{H}$ be a linear hypergraph with $r(\mathcal{H})\geq 3$. If $\Delta(\mathcal{H})$ is sufficiently large, then 
\begin{align*}
    \chii{\mathcal{H}}\leq 1.531r(\mathcal{H})\Delta(\mathcal{H}).
\end{align*}
\end{cor}

Mahdian \cite{Mahdian2000357} showed
$\chi'_s(G)\leq (2+o(1))\Delta(G)^2/\log\Delta(G)$ for $K_{2,2}$-free bipartite graphs $G$ with sufficiently large $\Delta(G)$,
and the bound is asymptotically best possible. This can be used to give another upper bound for the incidence chromatic number of linear hypergraphs.

Let $\mathcal{H}$ be a linear hypergraph. One can easily check that  $\mathcal{B}(\mathcal{H})$ is a $K_{2,2}$-free bipartite graph.
Hence applying \eqref{eq:relationship} we obtain the following.

\begin{cor}\label{cor:new2}
Let $\mathcal{H}$ be a linear hypergraph. If $\varrho(\mathcal{H}):=\max\{r(\mathcal{H}),\Delta(\mathcal{H})\}$ is sufficiently large, then 
\begin{align*}
    \chii{\mathcal{H}}\leq (2+o(1))\varrho(\mathcal{H})^2/\log\varrho(\mathcal{H}).
\end{align*}
\end{cor}

Comparing Corollary \ref{cor:new1} with \ref{cor:new2}, one can see that the bound given by Corollary \ref{cor:new1} is better than the one given by 
Corollary \ref{cor:new2} provided 
$r(\mathcal{H})\leq 1.3\Delta(\mathcal{H})/\log \Delta(\mathcal{H})$.

In 1990, Faudree, Gy\'arf\'as, Schelp, and Tuza \cite{Faudree1990205} conjectured $\chi'_s(G)\leq \Delta(G)^2$ for every bipartite graph $G$. In 1993, Brualdi and Quinn Massey \cite{Brualdi199351} refined it and put forward the following

\begin{conj}\label{conj:bipartite}
$\chi'_s(G)\leq ab$ for every $(a,b)$-bipartite graph.
\end{conj}

Nakprasit \cite{NAKPRASIT20083726} confirmed it for $a=2$.
Huang, Yu, and Zhou \cite{HUANG20171143} verified it for $a=3$ (there were some earlier partial results: Steger and Yu \cite{STEGER1993291} proved it for $a=b=3$, and Bensmail, Lagoutte, and Valicov \cite{BENSMAIL2016391}
proved $\chi'_s(G)\leq 4b$ for every $(3,b)$-bipartite graph $G$).
To our knowledge, whether Conjecture \ref{conj:bipartite} holds for $a=4$ is unknown.

Applying Theorem \ref{thm:limits} (or its detailed form Theorem \ref{maintheorem}), we 
obtain the following result towards Conjecture \ref{conj:bipartite}.

\begin{thm}
If $G$ is a $K_{2,t+1}$-free $(a,b)$-bipartite graph with $t<<a\leq b$, then 
\begin{align*}
    \chi'_s(G)\leq \bigg(\frac{4}{3}+o(1)\bigg)ab
\end{align*}
\end{thm}

\begin{proof}
Let $A$ and $B$ be the bipartition of $G$ with $\Delta(A)=a$ and $\Delta(B)=b$.
Let $\mathcal{H}$ be a hypergraph $(X,S)$ such that
$X=B$ and $S=\{N_G(u)~|~u\in A\}$.
One can see that $r(\mathcal{H})=a$, $\Delta(\mathcal{H})=b$, and $G=\mathcal{B}(\mathcal{H})$.
Since $G$ is $K_{2,t+1}$-free, $\mathcal{H}$ is $t$-quasi-linear. 
Hence by \eqref{eq:relationship} and by Theorem \ref{thm:limits} (or Theorem \ref{maintheorem}),
$\chi'_s(G)=\chii{\mathcal{H}}\leq \big(\frac{4}{3}+o(1)\big)ab$.
\end{proof}

We can also apply Theorem \ref{maintheorem} to obtain certain results in the following example form:

\begin{center}
    \textit{$\chi'_s(G)\leq 6b$ for every $K_{2,2}$-free $(4,b)$-bipartite graph with sufficiently large $b$.}
\end{center}

\noindent We leave the contents and proofs of them to the interested readers.

\section{$\mathbf{\alpha}$-Acyclic hypergraphs}\label{sec:6}

The \textit{minimization} $\mathcal{M}(\mathcal{H})$ of a hypergraph $\mathcal{H}=(X,S)$ is a hypergraph $(X',S')$ with 
$X'=X$ and $S'=\{e\in S~|~\forall f\in S, e \not\subset f\}$.
To begin with, we introduce results of 
Brault-Baron \cite{Brault-Baron2016} and
Fagin \cite{Fagin1983514}.

\begin{lem}\cite{Brault-Baron2016}\label{thm:alpha}
A hypergraph $\mathcal{H}=(X,S)$ is $\alpha$-acyclic if and only if there is no set $X'\subseteq X$ such that either $\mathcal{M}(\mathcal{H}[X'])$ is a usual graph cycle (i.e., a connected $2$-regular $2$-uniform hypergraph) or the edge set of $\mathcal{M}(\mathcal{H}[X'])$ is
$\{X'\setminus \{x\}~|~x\in X'\}$.
\end{lem}

\begin{lem}\cite{Fagin1983514}\label{thm:alpha2}
If $\mathcal{H}$ is a hypergraph such that $\mathcal{B}(\mathcal{H})$ is a forest, then $\mathcal{H}$ is $\alpha$-acyclic.
\end{lem}

\begin{lem}\label{thm:linear}
If $\mathcal{H}$ is an $\alpha$-acyclic linear hypergraph, then $\mathcal{B}(\mathcal{H})$ is a forest.
\end{lem}

\begin{proof}
Let $\mathcal{H}=(X,S)$ and $B(\mathcal{H})=(V_1,V_2; E)$ where $V_1=X,V_2=S$ and $E=\{xs~|~x\in X,s\in S,x\in s\}$.
Suppose, for a contradiction, that $\mathcal{B}(\mathcal{H})$ contains a cycle. We choose $C$ be the shortest cycle of $\mathcal{B}(\mathcal{H})$ and denote $C$ by $s_1x_1\cdots s_{q}x_{q}s_1$ ($q\geq 2$, $s_1\in S$).

If $q=2$, then $\{x_1,x_2\}\subseteq s_1\cap s_2$, contradicting the linearity of $\mathcal{H}$.
If $q\geq 3$, then let $X'=\{x_1,x_2,\ldots,x_q\}$.
For each $1\leq i\leq q$,
$s_{i}$ is the unique edge containing $\{x_{i-1},x_{i}\}$ (here we denote $x_0$ by $x_q$) by the linearity of $\mathcal{H}$. Hence by the minimum of $q$, $M(\mathcal{H}[X'])$ is a usual graph cycle, contradicting Lemma \ref{thm:alpha}.
\end{proof}

The following is an immediate corollary of Lemmas \ref{thm:alpha2} and \ref{thm:linear}.

\begin{cor}\label{cora}
If $\mathcal{H}$ is a linear hypergraph, then $\mathcal{H}$ is $\alpha$-acyclic if and only if $\mathcal{B}(\mathcal{H})$ is a forest.
\end{cor}

\begin{lem}\label{lvu}
If $\mathcal{H}$ is an $\alpha$-acyclic linear hypergraph and $\mathcal{H}'\subseteq \mathcal{H}$, then  $\mathcal{H}'$ is also an $\alpha$-acyclic linear hypergraph.
\end{lem}

\begin{proof}
Since $\mathcal{H}$ is $\alpha$-acyclic and linear, $\mathcal{B}(\mathcal{H})$ is a forest by Corollary \ref{cora}.
Since $\mathcal{H}'\subseteq \mathcal{H}$, $\mathcal{B}(\mathcal{H}')\subseteq \mathcal{B}(\mathcal{H})$ and thus 
$\mathcal{B}(\mathcal{H}')$ is a forest.
It is clear that $\mathcal{H'}$ is linear, and therefore it is $\alpha$-acyclic by Corollary \ref{cora}.
\end{proof}

Lemma \ref{lvu} is a key point of proving Theorem \ref{cort} by induction. To accomplish the proof of Theorem \ref{cort}, we need one more lemma as follows.

Given a strong edge coloring $\phi$ of a graph $G$, we use $\phi[v]$ denote the set of colors that are assigned to the edges incident with $v$.

\begin{lem}\label{lem:permution}
Let $T$ be a rooted tree with root $v$ and let $N_T(v)=\{u_1,\ldots,u_s\}$ such that $\D_T(u_i)\geq \D_T(u_j)$ whenever $i\geq j$.
We can modify any strong edge coloring  of $T$ by permuting the labels of the colors into a strong edge coloring $\varphi$ so that 
\begin{align*} 
\varphi[u_i]\setminus \varphi(vu_i) \supseteq \varphi[u_{i+1}]\setminus \varphi(vu_{i+1})
\end{align*}for each $1\leq i \leq s-1$.
\end{lem}

\begin{proof}
Given a strong edge coloring $\phi$ of $T$, let $j$ be the largest integer such that $1\leq j\leq s$ and $\phi$ 
can be modified by permuting the labels of the colors into a strong edge coloring $\varphi$ 
so that 
\begin{align}\label{eq}
\varphi[u_1]\setminus \varphi(vu_1) \supseteq \cdots \supseteq \varphi[u_{j}]\setminus \varphi(vu_{j}).
\end{align}
If $j=s$, then there is nothing to prove. We thus assume $j\leq s-1$.

We fix a $\varphi$ satisfying \eqref{eq} so that $Z_\varphi:=\big(\varphi[u_{j+1}]\setminus \varphi(vu_{j+1})\big)\setminus \big(\varphi[u_{j}]\setminus \varphi(vu_{j})\big)$ is as minimum as possible.
By the choice of $j$, $Z_\varphi\neq \emptyset$.
Let \[\alpha\in Z_\varphi.\]

Since $\deg_T(u_j)\geq \deg_T(u_{j+1})$ and $\varphi[u_{j}]\setminus \varphi(vu_{j}) \not\supseteq \varphi[u_{j+1}]\setminus \varphi(vu_{j+1})$, $\big(\varphi[u_{j}]\setminus \varphi(vu_{j})\big)\setminus \big(\varphi[u_{j+1}]\setminus \varphi(vu_{j+1})\big)$ is non-empty, and thus we let 
\[\beta\in \big(\varphi[u_{j}]\setminus \varphi(vu_{j})\big)\setminus \big(\varphi[u_{j+1}]\setminus \varphi(vu_{j+1})\big).\]

Since $\varphi$ is a strong edge coloring, $\varphi[v]\cap \{\alpha,\beta\}=\emptyset$. It guarantees that exchanging the colors of $\alpha$ and $\beta$ in the colored subtree $T_{u_{j+1}}$ induced by $u_{j+1}$ and its descendants would result in a strong edge coloring $\varphi'$ of $T$ such that
\[\varphi'[u_1]\setminus \varphi'(vu_1) \supseteq \cdots \supseteq \varphi'[u_{j}]\setminus \varphi'(vu_{j}),\]
and either
\begin{enumerate}[label=(\alph*)]
    \item\label{a} $1\leq |Z_{\varphi'}|<|Z_{\varphi}|$, or
    \item\label{b} $\varphi'[u_{j}]\setminus \varphi'(vu_{j}) \supseteq \varphi'[u_{j+1}]\setminus \varphi'(vu_{j+1})$.
\end{enumerate}
Note that \ref{a} contradicts the choice of $\varphi$ and \ref{b} contradicts the choice of $j$. This completes the proof.
\end{proof}

Now we are ready to prove Theorem \ref{cort} by the following two separating theorems.

\begin{thm}\label{ccxx}
If $\mathcal{H}$ is an $\alpha$-acyclic linear hypergraph and $\Delta,k$ are fixed integers such that $\Delta(\mathcal{H})\leq \Delta$ and $r(\mathcal{H})\leq k$, then $\chii{\mathcal{H}}\leq \Delta+k-1$.
\end{thm}

\begin{proof}
We proceed induction on the sum of the number of vertices and edges of $\mathcal{H}$. 
Let $\mathcal{H}_1,\ldots,\mathcal{H}_s$ ($s\geq 1$) be components of $\mathcal{H}$. 
If $s\geq 2$, then by Lemma \ref{lvu} and then by induction, $\chii{\mathcal{H}_j}\leq\Delta+k-1$ for $1\leq j\leq s$, and thus
$\chii{\mathcal{H}}=\max_{1\leq j\leq s}\chii{\mathcal{H}_j}\leq \Delta+k-1$.
So we assume $s=1$ below.

Since $\mathcal{H}$ is connected and $\alpha$-acyclic, $\mathcal{B}(\mathcal{H})$ is a tree by Lemma \ref{thm:linear}. Denote this tree by $T$ and root it at a leaf $r$. For every vertex $u$ of $T$, let $T_u$ be the subtree of $T$ induced by $u$ and its descendants.
Let $N_{T_r}(r)=\{z\}$ and $N_{T_{z}}(z)=\{u_1,u_2,\ldots,u_s\}$. Assume, without loss of generality, that $\deg_{T_{z}}(u_i)\geq \deg_{T_{z}}(u_{i+1})$ for each $1\leq i\leq s-1$. 
Let $\mathcal{H}'$ be the graph derived from $\mathcal{H}$ by removing $r$.
Now $T_{z}=\mathcal{B}(\mathcal{H}')$.

Note that $r$ is a vertex of $T$, representing either an edge or a vertex of $\mathcal{H}$.
If $r$ represents an edge (resp.\,a vertex) of $\mathcal{H}$, then $z$ is the unique vertex contained in (resp.\,edge containing) $r$ in $\mathcal{H}$, and $u_i$'s are edges containing (resp.\,vertices contained in) $z$.
It follows $\deg_T(z)\leq \Delta(\mathcal{H})\leq \Delta$ and $\deg_T(u_i)\leq k$ for each $1\leq i\leq s$ (resp.\,$\deg(z,T)\leq k$ and $\deg_T(u_i)\leq \Delta(\mathcal{H})\leq \Delta$  for each $1\leq i\leq s$).
By Lemma \ref{lvu} and then by induction, $\mathcal{H}'$ has a proper incidence $(\Delta+k-1)$-coloring, which can be translated into a strong edge $(\Delta+k-1)$-coloring $\phi$ of $T_{\ell}$ by Observation \ref{lem:bijection}. 

We permute by Lemma \ref{lem:permution} the labels of the colors of $\phi$ so that the resulting coloring $\varphi$ of $T_{\ell}$ satisfies that $\varphi[u_{i}]\setminus \varphi(vu_i) \supseteq \varphi[u_{i+1}]\setminus \varphi(vu_{i+1})$ for each $1\leq i \leq s-1$. 
Now we can finish a strong edge $(\Delta+k-1)$-coloring of $T$ by coloring the last uncolored edge $r\ell$ of $T$ with a color not in $\varphi[\ell]\cup \varphi[u_1]$. This is possible since $|\varphi[\ell]\cup \varphi[u_1]|\leq \Delta+k-2$ no matter $r$ represents an edge or a vertex of $\mathcal{H}$. Again, by Observation \ref{lem:bijection}, the strong edge $(\Delta+k-1)$-coloring of $T$ can be translated back to a proper incidence  $(\Delta+k-1)$-coloring of $\mathcal{H}$. Hence $\chi_\mathcal{I}(\mathcal{H})\leq\Delta+k-1$.
\end{proof}

\begin{thm}
If $\mathcal{H}$ is an $\alpha$-acyclic  $k$-uniform linear hypergraph, then $\chii{\mathcal{H}}=\Delta(\mathcal{H})+k-1$.
\end{thm}

\begin{proof}
There is nothing to be proved if $k=1$, so we assume $k\geq 2$.
Let $x_0$ be a vertex of $\mathcal{H}$ such that $\D_{\mathcal{H}}(x_0)=\Delta(\mathcal{H})$. Let $s_1,\ldots,s_{\Delta(\mathcal{H})}$ be the edges incident with $x_0$ and let  $s_1=\{x_0,x_1,x_2,\ldots,x_{k-1}\}$. Since the $\Delta(\mathcal{H})+k-1$ incidences $(x_0,s_1),\ldots,(x_0,s_{\Delta(\mathcal{H})}),$ $(x_1,s_1),\ldots,(x_{k-1},s_1)$ are pairwise adjacent, they cannot be colored the same. This  implies $\chii{\mathcal{H}}\geq\Delta(\mathcal{H})+k-1$
and thus the equality holds by Theorem \ref{ccxx}.
\end{proof}

\bibliographystyle{abbrv}
\bibliography{ref}

\end{document}